\documentclass{amsart}
\usepackage{all2018}
\usepackage{pxfonts}
\renewcommand{\phi}{\varphi}
\newcommand{\cI}{\mathcal{I}}
\renewcommand{\Vec}{\mathrm{Vec}}
\newcommand{\VecZ}{\mathrm{Vec}_\ZZ}
\newcommand{\comment}[1]{}
\DeclareMathOperator{\bigw}{\textstyle{\bigwedge}}

\begin{document}

\title{Polynomials and tensors of bounded strength}

\author{Arthur Bik}
\address{Universit\"at Bern, Mathematisches Institut,
Alpeneggstrasse 22,
3012 Bern, Switzerland}
\email{arthur.bik@math.unibe.ch}

\author{Jan Draisma}
\address{Universit\"at Bern, Mathematisches Institut,
Sidlerstrasse 5,
3012 Bern, Switzerland; and Eindhoven University of Technology}
\email{jan.draisma@math.unibe.ch}

\author{Rob H.~Eggermont}
\address{Eindhoven University of Technology, Department of
Mathematics and Computer Science, P.O.~Box 513, 5600 MB
Eindhoven, The Netherlands}
\email{r.h.eggermont@tue.nl}

\begin{abstract}
Notions of rank abound in the literature on tensor decomposition. We prove
that {\em strength}, recently introduced for homogeneous polynomials by
Ananyan-Hochster in their proof of Stillman's conjecture and generalised
here to other tensors, is universal among these ranks in the following
sense: any non-trivial Zariski-closed condition on tensors that is
functorial in the underlying vector space implies bounded strength. This
generalises a theorem by Derksen-Eggermont-Snowden on cubic polynomials,
as well as a theorem by Kazhdan-Ziegler which says that a polynomial all
of whose directional derivatives have bounded strength must itself have
bounded strength.
\end{abstract}

\thanks{The authors were supported by the NWO
Vici grant entitled {\em Stabilisation in Algebra and Geometry}.}
\maketitle

\section{Introduction and main theorem}

For any Zariski-closed affine cone $X$ that spans a vector space $V$,
the $X$-rank of $v \in V$ is the minimal number of terms across
all expressions of $v$ as a sum of vectors in $X$ \cite[Section
5.2.1]{Landsberg12}. If $X$ is the cone over a Veronese embedding of a
projective space, then the $X$-rank is the Waring rank; if it is the cone over
the Segre embedding of a product of projective spaces, then the $X$-rank is
the ordinary tensor rank; and if $X$ is the (reducible) variety of $d$-way
tensors that are a product of a vector with a $(d-1)$-way tensor, then
the $X$-rank is the slice rank \cite{Tao16}. Each of these ranks behaves functorially
in the underlying vector space(s) and is lower-bounded by the {\em
strength} that we will introduce below. In particular, in each of these
cases, bounded $X$-rank implies bounded strength.

This is not a coincidence. In this paper we establish that if the type
of tensors is fixed but the underlying vector space(s) are not, then
{\em any} non-trivial Zariski-closed condition that is functorial in the
underlying vector space(s) implies bounded strength. Note that bounded
$X$-rank is typically not a closed condition, but our result does apply
to its closure, bounded {\em border} $X$-rank.

We set the stage by discussing our result in detail for homogeneous
polynomials (symmetric tensors), and then treat the cases of alternating
and ordinary tensors more succinctly. We have chosen this case-by-case
treatment, rather than a more uniform treatment via polynomial functors,
to make the paper more immediately useful for researchers in tensor
decomposition. Moreover, in the symmetric and alternating cases there are
conditions on the characteristic of the ground field that are not needed
in the case of ordinary tensors; and finally, this case-by-case treatment
allows us to give explicit case-by-case bounds.

\subsection{Strength of homogeneous polynomials}
Let $K$ be a perfect and infinite field, write $\Vec=\Vec_K$ for the category of
finite-dimensional $K$-vector spaces, let $V \in \Vec$, and let $d \in
\ZZ_{\geq 2}$. We write $S^d V$ for the $d$-th symmetric power of $V$,
and we informally refer to elements of $S^d V$ as homogeneous polynomials
of degree $d$ in $\dim V$ variables.

This paper concerns decompositions of $q \in S^d V$ of the form
\begin{equation} \label{eq:Decomp} \tag{*}
q=r_1 s_1 + \cdots + r_k s_k
\end{equation}
where $r_i \in S^{e_i} V$ and $s_i \in S^{d-e_i} V$ for suitable natural
numbers $e_i \in \{1,\ldots,d-1\}$. The minimal number $k$ of terms
among all such decompositions of $q$ is called the {\em strength} $S(q)$
of $q$. This term was introduced in \cite{Hochster16}, except that we have taken
the liberty of adding $1$ to the strength defined there---so that, for
instance, our strength is subadditive. By taking all $e_i$ equal to $1$,
we obtain the bound $S(q) \leq \dim V$; but we will be interested in
upper bounds that do not depend on $\dim V$.

\subsection{Closed subsets of $S^d$} \label{ssec:ClosedPoly}
We equip $S^d V$ with the Zariski topology. A {\em closed subset} of $S^d$
is a rule $X$ that assigns to every $V \in \Vec$
a closed subset $X(V)$ of $S^d V$ in such a manner that for any linear
map $\phi \in \Hom_\Vec(V,W)$ the $d$-th
symmetric power $S^d \phi$ maps $X(V)$ into $X(W)$. In particular, if
we set $X_n:=X(K^n)$, then $X_n$ is stable under the group $\GL_n$,
and the map $S^d \pi: S^d K^n \to S^d K^{n-1}$ coming from the projection
$\pi: K^n \to K^{n-1}$ forgetting the last coordinate maps $X_n$ into
$X_{n-1}$---indeed, {\em onto} $X_{n-1}$, as one sees using the $d$-th
symmetric power of any section $K^{n-1} \to K^n$ of $\pi$.

\subsection{Examples}

\begin{ex}
Let $d=3$, fix $k \in \ZZz$, and let $Y_{3,k}(V) \subseteq S^d V$ be the
set of all elements of strength at most some fixed number $k$, i.e.,
the set of cubics that can be expressed as a sum of $k$ products of a
linear form and a quadric. In \cite{Derksen17} it is proved that $Y_{3,k}(V)$
is Zariski-closed. Since $S^d \phi$ maps a decomposition~\eqref{eq:Decomp} into
another such decomposition, $Y_{3,k}$ is a closed subset of $S^3$. \hfill
$\clubsuit$
\end{ex}

\begin{ex} \label{ex:Strength}
For arbitrary $d$, we do not know whether the set of elements in $S^dV$
of strength at most $k$ is closed. Let $Y_{d,k}(V)$ be its closure,
the set of elements in $S^d V$ of {\em border strength} at most $k$;
then $Y_{d,k}$ is a closed subset of $S^d$.  The parameterisation
\eqref{eq:Decomp} shows that the topological dimension of $Y_{d,k}(V)$
is at most a polynomial in $\dim V$ of degree $d-1$. As $\dim S^d V$ has
degree $d$ in $\dim V$, for $\dim V \gg 0$ the closed subset $Y_{d,k}(V)$
is not equal to $S^d V$.  This conclusion is abbreviated as $Y_{d,k}
\subsetneq S^d$.  \hfill $\clubsuit$
\end{ex}

Note that though the set of elements in $S^dV$ of strength at most $k$ is not necessarily closed, we will find (by means of Theorem~\ref{thm:Poly}) that its closure only contains elements of strength at most $N$ for some $N \geq k$ (provided we work over a field of characteristic $0$ or characteristic greater than $d$).

\begin{ex} \label{ex:KazhdanZiegler}
The paper \cite{Kazhdan18} concerns polynomials all of whose directional
derivatives have bounded strength. In our notation, this leads to the
closed set $Z_{d,k} \subseteq S^d$ defined by
\[ Z_{d,k}(V):=\left\{q \in S^d V ~\middle|~ \forall x \in
V^*:\  \langle x,q \rangle \in Y_{d-1,k}(V)\right\}, \]
where $\langle .,. \rangle : V^* \times S^d V \to S^{d-1} V$ is the
natural pairing. This set is potentially larger than the one in
\cite{Kazhdan18} since we allow directional derivatives to have strength
larger than $k$, as long as their border strength is at most $k$.

If, as in \cite{Kazhdan18}, $\cha K=0$ or $\cha K>d$, then for every
nonzero $x \in V^*$ the linear map $\langle x,. \rangle: S^d V \to
S^{d-1} V$ is surjective. Hence if $Y_{d-1,k}(V) \subsetneq S^{d-1} V$,
then also $Z_{d,k}(V) \subsetneq S^d V$. As the former happens for all $V$
with $\dim V \gg 0$, so does the latter, so $Z_{d,k} \subsetneq S^d$.
\hfill $\clubsuit$
\end{ex}

\subsection{Main theorem for polynomials}

Our main theorem is a converse to Example~\ref{ex:Strength}:
under mild conditions, in {\em any} closed $X \subsetneq S^d$ the strength
is uniformly bounded, independently of $\dim V$.

\begin{thm} \label{thm:Poly}
Fix $d \in \ZZ_{\geq2}$ and assume that $K$ is a perfect and infinite field
with $\cha K=0$ or $\cha K>d$. Then for any closed $X \subsetneq S^d$
there exists an $N$ such that for all finite-dimensional vector spaces
$V$ the strength of all elements in $X(V)$ is at most $N$.
\end{thm}

The theorem has a reformulation in terms of limits.  Set
$S^d_\infty:=\varprojlim S^d K^n$, the inverse limit under the maps
mentioned in \ref{ssec:ClosedPoly}.  The notion of strength extends
naturally to $S^d_\infty$, except that it can now be
infinite.

\begin{ex}
In \cite{Derksen17} it is proved that the polynomial
\[ f = x_1y_1z_1 + \dots + x_ny_nz_n \in S^3 K^{3n}, \]
where $x_1,y_1,z_1,\dots,x_n,y_n,z_n$ is the standard basis
of $K^{3n}$, has strength $n$, and that
\[ \sum_{i=1}^\infty x_iy_iz_i\in S^3_\infty \]
has infinite strength.
\hfill $\clubsuit$
\end{ex}

To find higher-degree examples of polynomials with high strength, we
have the following variant of an argument used in \cite{Derksen17}.

\begin{lm} \label{lm:Leibniz}
Suppose that $q \in S^d V$ has $S(q) \leq k$. Then for some $\ell \leq
k$ there exists an $\ell$-dimensional subspace $W \subseteq V$ such that the
image $\tilde{q}$ of $q$ in $S^d (V/W)$ has $S(\langle x,\tilde{q}
\rangle) \leq 2(k-\ell)$ for all $x \in (V/W)^*$.
\end{lm}

\begin{proof}
Arrange a decomposition $q=r_1 s_1 + \cdots + r_k s_k$ such that
$r_1,\ldots,r_\ell$ are linear and $r_j,s_j$ have degree at least $2$
for each $j > \ell$. Take $W=\langle r_1,\ldots,r_\ell \rangle_K$. Then we have
\[ \tilde{q}=\tilde{r}_{\ell+1} \tilde{s}_{\ell+1} + \cdots +
\tilde{r}_k \tilde{s}_k. \]
Now pair with any $x \in (V/W)^*$ and use the Leibniz rule
on the right to obtain $S(\langle x,\tilde{q}\rangle) \leq 2(k-\ell)$.
\end{proof}

\begin{ex}
Let $d \in \ZZ_{\geq 2}$. Let $x_1,\ldots,x_n$ be a basis of $\CC^n$ and
write $q_{d,n}:=x_1^d + \cdots + x_n^d$. Then we claim that $S(q_{d,n})
\to \infty$ for $n \to \infty$. First, since $q_{d,n-1}$ is obtained from
$q_{d,n}$ by setting $x_n$ equal to zero, $S(q_{d,n})_n$ is a weakly
increasing sequence, and it suffices to prove that it is unbounded. We
prove this by induction on $d$. For $d=2$ we have $S(q_{2,n})=\lceil
n/2 \rceil$ by elementary linear algebra. Let $d \geq 3$ and assume
the claim holds for $d-1$. Suppose $S(q_{d,n}) \leq k$,
and let $\ell \leq k$ and $W \subseteq \CC^n$ be as in
Lemma~\ref{lm:Leibniz}. Then
\[ \tilde{q}_{d,n}=\tilde{x}_1^d + \cdots + \tilde{x}_n^d \]
where, without loss of generality, we may assume that
$\tilde{x}_1,\ldots,\tilde{x}_{n-\ell}$ are a basis of $\CC^n/W$. It
follows that every directional derivative of
\[ \tilde{x}_1^d + \cdots + \tilde{x}_{n-\ell}^d
= \tilde{q}_{d,n} - (\tilde{x}_{n-\ell+1}^d + \cdots + \tilde{x}_n^d)
\]
has strength at most $2(k-\ell)+\ell=2k-\ell \leq 2k$. The left-hand
side is just $q_{d,n-\ell}$ in disguise. But $(\sum_i \pafg{x_i})
q_{d,n-\ell}=d q_{d-1,n-\ell}$, so $S(q_{d-1,n-k}) \leq S(
q_{d-1,n-\ell}) \leq 2k$. Hence if the sequence $(S(q_{d,n}))_n$ were
bounded from above, so was the sequence $(S(q_{d-1,n}))_n$, and this
contradiction proves the claim. An immediate consequence is that the
infinite series $x_1^d+x_2^d+\cdots$ has infinite strength. \hfill
$\clubsuit$
\end{ex}

Equip $S^d_\infty$ with the inverse limit topology.  The group
$\GL_\infty:=\bigcup_n \GL_n(K)$ acts on $S^d_\infty$ by homeomorphisms. The
following is an immediate consequence of Theorem~\ref{thm:Poly}.

\begin{cor} \label{cor:Limit}
Fix $d \in \ZZ_{\geq2}$ and assume that $K$ is an infinite and perfect field with
$\cha K=0$ or $\cha K>d$.  Then the $\GL_\infty$-orbit of
any element $q \in S^d_\infty$ of infinite strength is dense in~$S^d_\infty$.
\hfill $\square$
\end{cor}

\subsection{Strength of tensors and alternating tensors}

Let $V \in \Vec$. For alternating tensors $q \in \bigwedge^d V$, the strength $S(q)$
is defined as the minimal number $k$ of terms in any decomposition of the form
\[
q = r_1\wedge s_1 + \dots + r_k\wedge s_k
\]
where $r_i\in\bigwedge^{e_i}V$ and $s_i\in\bigwedge^{d-e_i} V$ for
suitable natural numbers $e_i \in \{1,\ldots,d-1\}$.  By taking all
$e_i$ equal to $1$, and using standard properties of the wedge product,
we obtain the bound $S(q)\leq\dim V - d+1$.

Next let $V=(V_1,\dots,V_d) \in \Vec^d$ and define $T^d V:=V_1
\otimes \cdots \otimes V_d$.  For tensors $q\in T^d V$,
the strength $S(q)$ is defined as the minimal number $k$ of terms in any
decomposition of the form
\[
q=r_1\otimes s_1 + \dots + r_k\otimes s_k
\]
where $r_i\in\bigotimes_{j\in J_i}V_j$ and
$s_i\in\bigotimes_{j\in[d]\setminus J_i} V_j$ for suitable non-empty
subsets $J_i\subsetneq[d]$. By taking all $J_i$ equal to $\{\ell\}$,
we obtain the bound $S(q)\leq\dim V_{\ell}$ for any $\ell\in[d]$.

\subsection{Closed subsets of $\Wedge^d$ and $T^d$}

For all $V\in\Vec$, we equip $\bigwedge^dV$ with the Zariski topology. A
closed subset of $\bigwedge^d$ is a rule $X$ that assigns to every
$V\in\Vec$ a closed subset $X(V)$ of $\bigwedge^d V$ in such a
manner that for any $\phi \in \Hom_\Vec(V,W)$ the alternating power
$\bigwedge^d\phi$ maps $X(V)$ into $X(W)$. In particular, the subset
$X(V)$ is stable under the group $\GL(V)$ and if the map $\phi$ is
surjective, then $X(V)$ maps onto $X(W)$.

Similarly, for $V=(V_1,\dots,V_d)\in\Vec^d$, we equip $T^dV$
with the Zariski topology. A closed subset of $T^d$ is
a rule $X$ that assigns to every $V\in\Vec^d$ a closed subset
$X(V)$ of $T^dV$ in such a manner that for any $d$-tuple of
linear maps $\phi_i \in \Hom_{\Vec}(V_i,W_i)$ the tensor product
$T^d\phi:=\phi_1\otimes\dots\otimes\phi_d$ maps $X(V)$ into
$X(W)$. In particular, the subset $X(V)$ is stable under the group
$\GL(V):=\GL(V_1)\times\dots\times\GL(V_d)$ and if the maps $\phi_i$ are all
surjective, then $X(V)$ maps onto $X(W)$.

\subsection{Main theorem for tensors and alternating tensors}

\begin{thm}\label{thm:AltTensor}
Fix $d\in\ZZ_{\geq2}$ and let $K$ be a perfect and infinite field with $\cha K=0$ or $\cha K>d$.
Then for any closed $X\subsetneq \bigwedge^d$ there exists an $N$ such that for
all finite-dimensional vector spaces $V$ the strength of all elements
in $X(V)$ is at most $N$.
\end{thm}

\begin{thm}\label{thm:Tensor}
Fix $d\in\ZZ_{\geq2}$ and let $K$ be a perfect and infinite field. Then
for any closed $X\subsetneq T^d$ there exists an $N$ such that for all
finite-dimensional vector spaces $V_1,\dots,V_d$ the strength of all
elements in $X(V_1,\dots,V_d)$ is at most $N$.
\end{thm}

Both of these theorems have reformulations in terms of suitable projective
limits as in Corollary~\ref{cor:Limit}; we leave these to the reader.

\subsection{A version over $\ZZ$}

Theorems~\ref{thm:Poly},~\ref{thm:AltTensor}, and~\ref{thm:Tensor}
require that $K$ be fixed in advance, and allow for the closed subsets of
$S^d,\Wedge^d,T^d$ to be defined by equations specific to $K$. The price
that we pay for this generality is that we need to require $K$ to be perfect
and infinite, and that the values of $N$ in these theorems depend on $K$.

Indeed, in the proofs, perfectness of the field is used to ensure that
a squarefree nonzero polynomial has some nonzero directional derivative;
and infiniteness of the field is used to ensure that if some polynomial
in $t$ vanishes for all $t \in K$, then the coefficients of all monomials
$t^d$ vanish. We can get around both of these restrictions by working
only with tensor properties defined over $\ZZ$ before specialising to $K$.

Let $\VecZ$ be the category of finite-rank free $\ZZ$-modules with
$\ZZ$-linear maps. Every object $V \in \VecZ$ gives rise to an affine scheme, the
spectrum of the symmetric algebra (over $\ZZ$) on the module dual to $V$.
By abuse of notation, we write $V$ for this scheme, as well. The scheme
of a product $V \times W$ is canonically isomorphic to the product of
the schemes, and a $\phi \in \Hom_{\VecZ}(V,W)$ determines a morphism
of schemes $V \to W$.

A module $V \in \VecZ$ has a symmetric power $S^d_\ZZ V \in \VecZ$
characterised by the usual universal property.  A {\em closed subscheme}
of $S^d_\ZZ$ is a rule $X_\ZZ$ that assigns to each $V \in \VecZ$
a closed subscheme of $S^d_\ZZ V$ in such a manner that for $V,W \in
\VecZ$ and $\phi \in \Hom_{\VecZ}(V,W)$ the morphism $S^d_\ZZ \phi$
maps $X_\ZZ(V)$ into $X_\ZZ(W)$. This is equivalent to the condition
that the morphism of schemes determined by
\[ S^d_\ZZ V \times \Hom_{\VecZ}(V,W) \to S^d_\ZZ W, \quad
(v_1 \cdots v_d,\phi) \mapsto \phi(v_1) \cdots \phi(v_d)
\]
maps $X_\ZZ(V) \times \Hom_{\VecZ}(V,W)$ into $X_\ZZ(W)$.

In terms of equations this means the following: Suppose
that $V=\ZZ^m$ and $W=\ZZ^n$, let $f$ be any polynomial in the
$\binom{n-1+d}{d}$ standard coordinates on $S^d_\ZZ W$ with coefficients
in $\ZZ$, and let $\phi$ be an $n \times m$-matrix whose entries
$\phi_{ij}$ are variables. Then one can expand $f \circ S^d_\ZZ \phi$ as
a polynomial $\sum_{\alpha \in \ZZz^{n \times m}} c_\alpha \phi^\alpha$
in the $\phi_{ij}$ whose coefficients $c_\alpha$ are polynomials in the
$\binom{m-1+d}{d}$ standard coordinates on $S^d_\ZZ V$.  The condition
above says that if $f$ is in the ideal of $X_\ZZ(W)$, then
all the $c_\alpha$ lie in the ideal of $X_\ZZ(V)$.

If $X_\ZZ$ is a closed subscheme of $S^d_\ZZ$, then for each field $K$
we obtain a closed subset $X_K$ of $S^d=S^d_K$ as follows: for $V \in
\Vec=\Vec_K$ choose any linear isomorphism $\phi:V \to K^n$, and let
$X_K(V)$ be the preimage under $S^d \phi$ of the set of $K$-valued points
of the scheme $X(\ZZ^n) \subseteq S^d \ZZ^n$.

\begin{re}\label{re:XoverZ}
For all field extensions $K\subseteq L$, we have
$$
X_L(V\otimes_KL)\cap S^dV=X_K(V)
$$
for all vector spaces $V\in\Vec_K$.
\end{re}

\begin{thm} \label{thm:PolyZ}
Let $d \in \ZZ_{\geq 2}$ and let $X_\ZZ$ be a closed subscheme of
$S^d_\ZZ$. Then there exists an $N \in \ZZz$ such that the following
holds:
\begin{itemize}
\item[(\dag)] Let $K$ be any field with $\cha K=0$ or $\cha K>d$ such that
$X_K \subsetneq S^d_K$. Then for all $V \in \Vec_K$ the
strength of all elements in $X_K(V)$ is at most $N$.
\end{itemize}
\end{thm}

\begin{ex}[Examples~\ref{ex:Strength} and~\ref{ex:KazhdanZiegler} revisited]
Let $V \in \VecZ$. For a sequence $e_1,\ldots,e_k$ of integers in
$\{1,\ldots,d-1\}$ there is a morphism of schemes
\[ \prod_{i=1}^k (S^{e_i}_\ZZ V \times S^{d-e_i}_\ZZ V) \to
S^d_\ZZ V,\quad
((r_1,s_1),\ldots,(r_k,s_k)) \mapsto r_1 s_1 + \cdots + r_k
s_k. \]
We define $Y_{d,k,\ZZ}(V)$ as its scheme-theoretic image, i.e., as the
closed subscheme defined by the kernel of the homomorphism of rings
in the opposite direction. A straightforward verification shows that
$Y_{d,k,\ZZ}$ is a closed subscheme of $S^d_\ZZ$. A similar construction
yields a $\ZZ$-version $Z_{d,k,\ZZ}$ of $Z_{d,k}$. The previous theorem
applied to these closed subschemes yields a bound $N$ on the strength of
elements of $Z_{d,k}(W),\ W \in \Vec_K$ that is independent of $W$ and $K$
($K$ may now be finite or non-perfect, but must satisfy the conditions
on the characteristic). This yields the main result of \cite{Kazhdan18}.
\end{ex}

The $\ZZ$-constructions in this subsection have analogues
for the polynomial functors $\Wedge^d$ and $T^d$, and the analogues of
Theorems~\ref{thm:AltTensor} and~\ref{thm:Tensor} also hold over $\ZZ$.

\begin{thm} \label{thm:AltTensorZ}
Let $X$ be a closed subscheme of $\Wedge^d_\ZZ$. Then there exists an
$N \in \ZZz$ such that the following holds:
\begin{itemize}
\item[(\dag)]  Let $K$ be any field with $\cha K=0$ or $\cha K>d$ such that
$X_K \subsetneq \Wedge^d_K$. Then for all $V \in \Vec_K$ and all
$q \in X_K(V)$ we have $S(q) \leq N$.
\end{itemize}
\end{thm}

\begin{thm} \label{thm:TensorZ}
Let $X$ be a closed subscheme of $T^d_\ZZ$. Then there exists an $N \in \ZZz$
such that the following holds:
\begin{itemize}
\item[(\dag)]  Let $K$ be any field such that $X_K \subsetneq T^d_K$. Then
for all $V \in \Vec_K^d$ and all $q \in X_K(V)$ we have $S(q) \leq N$.
\end{itemize}
\end{thm}

\subsection{Relation to previous work}

For $d=3$, our Corollary~\ref{cor:Limit} is proved in \cite[Theorem
1.7]{Derksen17}. For $Z_{d,k,\ZZ}$, our Theorem~\ref{thm:Poly} (and
indeed, reading carefully, Theorem~\ref{thm:PolyZ}) is proved in
\cite{Kazhdan18} using Gowers norms---interestingly, this requires a
detour via finite fields to prove the theorem for $K=\CC$. Our
proof is quite different, and follows the second author's technique
from \cite{Draisma17}.

Unlike for Waring rank and tensor rank, there seems to be very little
literature on equations for the varieties $Y_{d,k}(V)$ of polynomials
of bounded strength. For $d=3$ and $k=1$ the ideal is generated by 35
polynomials of degree $8$ \cite{Chipalkatti02}.

\subsection{Organisation}
In Section~\ref{sec:Poly} we prove Theorem~\ref{thm:Poly}.
In Sections~\ref{sec:AltTensor} and~\ref{sec:Tensor} we adapt
this proof to the case of alternating and ordinary tensors,
respectively. In Remarks~\ref{re:Explicit}, \ref{re:ExplicitAlt},
and~\ref{re:ExplicitTensor} we give explicit (though probably
non-optimal!) values for $N$ from Theorems~\ref{thm:Poly},
\ref{thm:AltTensor}, and~\ref{thm:Tensor}. These are used in
Section~\ref{sec:PolyZ}, where we prove Theorem~\ref{thm:PolyZ}.

\subsection*{Acknowledgments}
The authors thank Professor Kazhdan for the heads-up on \cite{Kazhdan18}
and stimulating e-mail discussions. The second author thanks the Institut
Mittag-Leffler for their hospitality during this project.

\section{Proof for polynomials} \label{sec:Poly}
In this section we prove Theorem~\ref{thm:Poly}. By assumption, there
exists a $U \in \Vec$ such that $X(U) \subsetneq S^d U$. We fix this $U$
throughout the proof. The bound $N$ that we will obtain depends only on
$d$ and $\dim U$, see Remark~\ref{re:Explicit}.

\subsection{Irreducibility}
The following lemma is a standard fact from representation theory. Recall that
$\cha K=0$ or $\cha K>d$. This implies that for any $V \in \Vec$ the
$\GL(V)$-module $S^d(V^*)$ is isomorphic to $(S^d V)^*$.

\begin{lm} \label{lm:Irred}
For each $V \in \Vec$, the $\GL(V)$-module $S^d V$ is irreducible and
linearly spanned by its closed subvariety $P:=\{v^d \mid v \in V\}
\subseteq S^d V$. Furthermore, any $\GL(V)$-equivariant polynomial
map from $V$ into a $\GL(V)$-module $N$ on which $t 1_V$ acts via
multiplication with~$t^d$ factors as $V \to S^d V,\  v \mapsto v^d$
and a unique $\GL(V)$-equivariant linear map $S^dV \to N$.\hfill\qed
\end{lm}

\subsection{Homogeneity}
We equip the coordinate ring $K[S^d V]$ with the $d \cdot \ZZz$-grading
in which the elements of $(S^d V)^*$ have degree $d$. For any closed $X
\subseteq S^d$ we find, from the fact that $X(V)$ is $\GL(V)$-stable,
that the ideal $\cI(X(V)) \subseteq K[S^d V]$ is $\GL(V)$-stable and in
particular homogeneous. We define
\[ \delta_X:=\min_{f \in \cI(X(U)) \setminus \{0\}} \deg f \in d \ZZz. \]

\subsection{Induction}
If $\delta_X=0$, then $\cI(X(U))$ contains a nonzero constant polynomial,
so that $X(U)=\emptyset$. For any $V \in \Vec$, the $d$-th symmetric
power of the zero map $0_{V \to U}:V \to U$ maps $X(V)$ into $X(U)$,
hence all $X(V)$ are empty. Hence in Theorem~\ref{thm:Poly} we may take
$N=0$. We proceed by induction, assuming that $\delta_X>0$ and that the
theorem holds for all $Y \subseteq S^d$ with $Y(U) \subsetneq S^d U$
and $\delta_Y < \delta_X$.

\subsection{Derivative}

Let $f \in \cI(X(U)) \setminus \{0\}$ be homogeneous of degree $\delta_X$.
By minimality of $\delta_X$ and perfectness of $K$, there exists an
$r \in S^d U$ such that the directional derivative $h:=\frac{\partial
f}{\partial r}$ is not the zero polynomial. By Lemma~\ref{lm:Irred}, $S^d
U$ is spanned by $d$-th powers, so we may further assume that $r=u^d$
for some $u \in U$.

We define the closed subset $Y \subsetneq S^d$ by
\[ Y(V):=\left\{q \in X(V) ~\middle|~ \forall \phi \in \Hom_\Vec(V,U):
 h((S^d \phi) q) =0 \right\}. \]
Now $\delta_Y \leq \deg h=\deg f-d<\deg f$, so by the induction hypothesis
the theorem holds for $Y$. We define
\[ Z(V):=X(V) \setminus Y(V), \]
and set out to prove that all elements in $Z(V)$ have bounded strength
independent of $V$.

\subsection{Shifting}
For $V \in \Vec$ we define
\begin{align*}
P'(V)&:=S^d (U \oplus V)=\bigoplus_{i=0}^d S^{d-i} U \otimes S^i V, \\
X'(V)&:=X(U \oplus V) \subseteq P'(V), \\
Z'(V)&:=\left\{q \in X'(V) ~\middle|~ h((S^d(1_U \oplus 0_{V \to U})) q) \neq 0\right\};
\end{align*}
the notation is chosen compatible with \cite{Draisma17}. We think of
$P'(V),X'(V)$ as varieties over $S^d U,X(U)$, respectively, via the
linear map $S^d(1_U \oplus 0_{V \to U})$. Accordingly, by slight abuse
of notation, we will write $h$ for $h \circ S^d(1_U \oplus 0_{V \to U})$.

\begin{lm}
We have
\[ Z(U \oplus V) = \bigcup_{g \in \GL(U \oplus V)} g Z'(V). \]
In particular, $\sup_{q \in Z(U \oplus V)} S(q) = \sup_{q
\in Z'(V)} S(q)$.
\end{lm}
\begin{proof}
First, we have $Z(U \oplus V) \supseteq Z'(V)$, and since the left-hand
side is $\GL(U \oplus V)$-stable, the inclusion $\supseteq$ follows.
Conversely, if $q \in Z(U \oplus V)$, then there exists a linear map
$\phi: U \oplus V \to U$ for which $h((S^d \phi) q) \neq 0$.  Since this
is an open condition on $\phi$, we may further assume that $\phi$ has
full rank. Then for a suitable $g \in \GL(U \oplus V)$ we find that
$\phi =(1_U \oplus 0_{V \to U}) \circ g$.  Accordingly,
\[ h(gq) = h((S^d (1_U \oplus 0_{V \to U}))(S^dg) q) = h((S^d\phi) q) \neq 0 \]
and hence $gq \in Z'(V)$.
\end{proof}

\begin{lm} \label{lm:XYZ}
We have
\[ \sup_{V \in \Vec} \sup_{f \in X(V)} S(f) = \sup_{V \in
\Vec} \sup\left\{\sup_{f \in Y(V)} S(f), \sup_{f \in Z'(V)}
S(f)\right\}. \]
\end{lm}

\begin{proof}
The same statement with $Z'(V)$ replaced by $Z(V)$ is obvious given the fact that
$X(V)=Y(V)\cup Z(V)$. The map $S^d(0_{V\to U},1_V)$ maps $Z(V)$ into
$Z(U\oplus V)$ and is easily seen to preserve the strength. By the previous lemma
\[\sup_{f \in Z'(V)}S(f)=\sup_{f\in Z(U\oplus V)}S(f)\]
and so the statement follows.
\end{proof}

So it suffices to show that elements in $Z'(V)$ have bounded strength.

\subsection{Chopping}

\begin{lm} \label{lm:Chopping}
For $q \in P'(V)$ write $q=q_0 + \cdots + q_d$ with $q_i \in S^{d-i}
U \otimes S^i V$. Then
\[ S(q) \leq \dim U + S(q_d). \]
\end{lm}
\begin{proof}
Note that $q_0+\ldots+q_{d-1}$ is in the image of the map
\begin{eqnarray*}
U\otimes S^{d-1}(U\oplus V)&\to& S^d(U\oplus V)\\
r\otimes s&\mapsto&rs
\end{eqnarray*}
and hence has strength at most $\dim U$. Now, strength is subadditive, so
\[ S(q) \leq S(q_0 + \ldots + q_{d-1}) + S(q_d)\leq \dim U+S(q_d). \qedhere\]
\end{proof}

So, as $U$ is fixed, it suffices to prove that for $V$ ranging
through $\Vec$ and $q$ ranging through $Z'(V)$ the component $q_d$
has bounded strength.

\subsection{Embedding}

Define
\[ Q'(V):=P'(V)/S^d V = \bigoplus_{i=0}^{d-1} S^{d-i} U \otimes S^i V \]
and write $\pi(V):P'(V) \to Q'(V)$ for the natural projection. Let
\[ B(V):=\{q \in Q'(V) \mid h(q) \neq 0\}=\{(q_0,\dots,q_{d-1})\in Q'(V) \mid h(q_0) \neq 0\}, \]
an open set in $Q'(V)$.  Then $\pi(V)$ maps $Z'(V)$ into $B(V)$, and by
\cite[Lemma 7]{Draisma17} and Lemma~\ref{lm:Irred}, we have the following.

\begin{lm} \label{lm:Embedding}
The map $\pi(V)$ restricts to a closed embedding $Z'(V) \to
B(V)$. \hfill $\square$
\end{lm}

We will not actually use this lemma, but we will use its proof method.

\subsection{An equivariant map back} \label{ssec:Psi}

We construct a suitable map opposite to the embedding of
Lemma~\ref{lm:Embedding}.

\begin{lm} \label{lm:Psi}
There exists a $\GL(V)$-equivariant polynomial map $\Psi:Q'(V) \to
S^dV$ such that, for all $q=(q_0,\ldots,q_d) \in Z'(V)$, $q_d$ is a
scalar multiple of $\Psi(q_0,\ldots,q_{d-1})$.
\end{lm}

\begin{proof}
For $x \in V^*$, let $\phi_x:V \to U$ be the linear map $v \mapsto x(v) u$;
here $u$ is the vector used in the definition of $h$. Note that $x \mapsto
\phi_x$ is a $\GL(V)$-equivariant linear map $V^* \to \Hom_\Vec(V,U)$.

For every $t \in K$ define the linear map $\Phi_x(t):=
S^d(1_U \oplus t\phi_x): P'(V) \to S^d U$. The restriction of $\Phi_x(t)$
to the summand $S^{d-i} U \otimes S^i V$ equals $t^i \Phi_{x,i}$
where $\Phi_{x,i}$ is the composition of $S^{d-i} 1_U \otimes S^i
\phi_x: S^{d-i} U \otimes S^i V \to S^{d-i} U \otimes S^i U$ and the
multiplication map into $S^d U$. In particular, $\Phi_{x,0}$ is the
identity on $S^d U$ and $\Phi_{x,d}$ is the linear map $S^d V \to S^d U,
q_d \mapsto x^d(q_d) u^d$. Note that $V^* \to \Hom_\Vec(S^{d-i} U \otimes
S^i V, S^d U), x \mapsto \Phi_{x,i}$ is a $\GL(V)$-equivariant polynomial
map of (ordinary) degree $i$.

The functoriality of $X$ implies that $\Phi_x(t) X'(V) \subseteq X(U)$. In
particular, the pull-back of $f$ along $\Phi_x(t)$ to $P'(V)$ vanishes
on $X'(V)$. Take $q=(q_0,\ldots,q_d) \in P'(V)$. Then
\[ f(\Phi_x(t)(q_0 + q_1 + \cdots + q_{d-1} + q_d))\\
= f\left(q_0 + t \Phi_{x,1} q_1 + \cdots t^{d-1} \Phi_{x,d-1} q_{d-1} +
t^d x^d(q_d) u^d\right), \]
and this vanishes for $q \in X'(V)$. In particular, the coefficient of
$t^d$ in the Taylor expansion of this expression vanishes
for $q \in X'(V)$ . This coefficient equals
\[ x^d(q_d) \pafg[f]{u^d}(q_0) +
\Psi(x,q_0,\ldots,q_{d-1})= x^d(q_d) h(q_0) + \Psi(x,q_0,\ldots,q_{d-1}). \]
where the function $\Psi: V^* \times Q'(V) \mapsto K$ is
$\GL(V)$-invariant, and homogeneous of degree $d$ in its first argument
$x$.

For $q \in Z'(V)$ we have $h(q_0) \neq 0$ and hence
\[ x^d(q_d) = - \frac{1}{h(q_0)} \Psi(x,q_0,\ldots,q_{d-1}). \]
By Lemma~\ref{lm:Irred} the space $(S^d V)^*$ of coordinates on $S^d$
is spanned by the $x^d, x \in V^*$, and this shows that $Z'(V) \to B(V)$
is a closed embedding. But it yields more: by Lemma~\ref{lm:Irred}, $\Psi$ factors as
\begin{eqnarray*}
V^*\times Q'(V) &\to& S^d V^* \times Q'(V)\\
(x,q') &\mapsto& (x^d,q')
\end{eqnarray*}
and a unique $\GL(V)$-invariant map $S^d V^* \times Q'(V) \to K$. We
denote the latter map also by $\Psi$, which is now linear in its first
argument. If we re-interpret $\Psi$ as $\GL(V)$-equivariant polynomial
map $Q'(V) \to S^d V$, then for $q \in Z'(V)$ we have
\[ q_d = - \frac{1}{h(q_0)} \Psi(q_0,\ldots,q_{d-1}). \]
In particular, for $q \in Z'(V)$ we have $q_d \in \im \Psi$.
\end{proof}

\subsection{Covariants}

A {\em covariant} of $Q'(V)$ (of order $S^d V$, the only order that
we will use) is a $\GL(V)$-equivariant polynomial map $Q'(V) \to
S^d V$. So the map $\Psi$ constructed in~\ref{ssec:Psi}
is a covariant. For each integer $i\in[d-1]$, choose a basis $u_{i,1},\dots,u_{i,n_i}$ of $S^{d-i}U$. Then the map
\begin{eqnarray*}
\Phi\colon \bigoplus_{i=0}^{d-i}\left(S^iV\right)^{\oplus n_i}&\to&\bigoplus_{i=1}^{d-1} S^{d-i} U \otimes S^i V\\
(w_{i,j})_{i,j}&\mapsto&\left(\sum_{j=1}^{n_i}u_{i,j}\otimes w_{i,j}\right)_{i=1}^{d-1}
\end{eqnarray*}
is a $\GL(V)$-equivariant isomorphism and the following lemma holds.

\begin{lm} \label{lm:Covariants}
Let $\Psi\colon Q'(V)\to S^dV$ be a covariant. Then the composition $\Psi\circ(\id_{S^dU},\Phi)$ is of the form
\begin{eqnarray*}
S^dU\oplus\bigoplus_{i=1}^{d-1}\left(S^iV\right)^{\oplus n_i}&\to&S^dV\\
\left(q,(w_{i,j})_{i,j}\right)&\mapsto&\sum_{\substack{\alpha_{i,j}\in\ZZ_{\geq0}\\\sum_{i=1}^{d-1}\sum_{j=1}^{n_i}i\cdot\alpha_{i,j}=d}}p_{\alpha}(q)\cdot \prod_{i=1}^{d-1}\prod_{j=1}^{n_i}w_{i,j}^{\alpha_{i,j}}
\end{eqnarray*}
for some polynomial functions $p_{\alpha}\colon S^dU\to K$.
\end{lm}
\begin{proof}
Polynomial $\GL(V)$-equivariant maps
$$
S^dU\oplus\bigoplus_{i=1}^{d-1}\left(S^iV\right)^{\oplus n_i}\to S^dV
$$
correspond one-to-one to linear $\GL(V)$-equivariant maps
$$
K[S^dU]\otimes\bigoplus_{\substack{\alpha_{i,j}\in\ZZ_{\geq0}\\\sum_{i=1}^{d-1}\sum_{j=1}^{n_i}i\cdot\alpha_{i,j}=d}}\bigotimes_{i=1}^{d-1}\bigotimes_{j=1}^{n_i}S^{\alpha_{i,j}}S^iV\to S^dV
$$
by the universal properties of tensor products and symmetric powers. For each $\alpha$, the vector space
$$
\Hom_{\GL(V)}\left(\bigotimes_{i=1}^{d-1}\bigotimes_{j=1}^{n_i}S^{\alpha_{i,j}}S^iV,S^dV\right)
$$
is one-dimensional and consists of multiples of the homomorphism $\ell_{\alpha}$ sending 
$$
\bigotimes_{i=1}^{d-1}\bigotimes_{j=1}^{n_i}w_{i,j,1}\cdot\dots\cdot w_{i,j,\alpha_{i,j}}\mapsto \prod_{i=1}^{d-1}\prod_{j=1}^{n_i}w_{i,j,1}\cdot\dots\cdot w_{i,j,\alpha_{i,j}}.
$$
Hence the set of linear $\GL(V)$-equivariant maps
$$
K[S^dU]\otimes\bigoplus_{\substack{\alpha_{i,j}\in\ZZ_{\geq0}\\\sum_{i=1}^{d-1}\sum_{j=1}^{n_i}i\cdot\alpha_{i,j}=d}}\bigotimes_{i=1}^{d-1}\bigotimes_{j=1}^{n_i}S^{\alpha_{i,j}}S^iV\to S^dV
$$
is spanned as a $K[S^dU]$-module by the set of all $\ell_{\alpha}$. The corresponding statement for polynomial $\GL(V)$-equivariant maps is the statement of the lemma. 
\end{proof}

\subsection{Conclusion of the proof}

\begin{proof}[Proof of Theorem~\ref{thm:Poly}]
By the induction hypothesis and Lemma~\ref{lm:XYZ}, to bound the
strength of elements of $X(V)$ for all $V \in \Vec$ it suffices to
bound the strength of elements of $Z'(V)$ for all $V \in \Vec$. By
Lemma~\ref{lm:Chopping}, it suffices to bound the strength of
$q_d$ over all $q=(q_0,\ldots,q_d) \in Z'(V)$. By Lemma \ref{lm:Psi}, 
we know that such a $q_d$ is contained in the image of a covariant. 
So using Lemma \ref{lm:Covariants}, we see that $q_d$ is a linear 
combination of products of polynomials $w_{i,j}\in S^iV$ where $i$ ranges over $[d-1]$ 
and $j$ ranges over $[\dim S^{d-i}U]$. Since each of those products has degree $d$ and since, for each pair $(i,j)$, the polynomial $w_{i,j}$ has degree $i$, we see that each of the products is divisible by $w_{i,j}$ for some $i\leq d/2$. We find that the strength of $q_d$ is at most
$$
\#\{w_{i,j}\mid i\in\{1,\dots,\lfloor d/2\rfloor\},j\in[\dim S^{d-i}U]\}\leq\sum_{i=1}^{\lfloor d/2\rfloor}\dim S^{d-i}U
$$
and this bounds the strength of $q_d$ independently of $V$.
\end{proof}

\begin{re} \label{re:Explicit}
It follows from the induction that $N$ from Theorem~\ref{thm:Poly}
can be taken equal to
\begin{equation} \label{eq:Bound}
\dim U +\sum_{i=1}^{\lfloor d/2\rfloor}\dim S^{d-i}U.
\end{equation}
\end{re}

\section{Proof for alternating tensors} \label{sec:AltTensor}

In this section we adapt the proof from Section~\ref{sec:Poly} to a
proof of Theorem~\ref{thm:AltTensor} for alternating tensors.
Throughout this section, we assume that $\cha K=0$ or $\cha K>d$, we will let $X$
be a closed subset of $\Wedge^d$ and $U$ a finite-dimensional
 vector space such that $X(U) \subsetneq \Wedge^dU$. Note that
 $\dim U\geq d$ as $\bigwedge^dU\neq\emptyset$.

\subsection{Irreducibility}

Note that for any $V\in\Vec$ the $\GL(V)$-modules $\Wedge^d(V^*)$ and $(\Wedge^d V)^*$ are isomorphic. The analogue of Lemma~\ref{lm:Irred} is as follows.

\begin{lm} \label{lm:AltIrred} For each $V \in \Vec$, the $\GL(V)$-module
$\bigwedge^dV$ is irreducible and linearly spanned by its closed subvariety
$P := \{v_1\wedge\ldots\wedge v_d \mid v_1,\ldots,v_d \in V\mbox{ linearly independent}\}
 \subseteq \bigwedge^dV$. Furthermore, any $\GL(V)$-equivariant multilinear and alternating map
 from $V^d$ into a $\GL(V)$-module $N$ on which $t 1_V$ acts via
multiplication with $t^d$ extends uniquely to a
$\GL(V)$-equivariant linear map $\bigwedge^dV \to N$.\hfill $\square$
\end{lm}

\subsection{Homogeneity}
We equip the coordinate ring $K[\bigwedge^dV]$ with the $d\cdot\ZZz$-grading
 in which the elements of $(\bigwedge^dV)^*$ have degree $d$. For any closed $X\subseteq
\bigwedge^d$ we find, from the fact that $X(V)$ is stable under $\GL(V)$,
that the ideal $\cI(X(V))\subseteq K[\bigwedge^dV]$ is $\GL(V)$-stable
and in particular homogeneous. We define
$$
\delta_X:=\min_{f\in\cI(X(U))\setminus\{0\}}\deg f \in d\ZZz.
$$

\subsection{Induction}
If $\delta_X=0$, then we find that $X(V)=\emptyset$ for all $V\in\Vec$. We may
therefore assume that $\delta_X>0$ and we proceed by induction,
assuming that the theorem holds for all $Y\subseteq \bigwedge^d$ with $Y(U)\subsetneq\Wedge^dU$ and $\delta_Y < \delta_X$.

\subsection{Derivative}
Let $f\in \cI(X(U))\setminus\{0\}$ be a homogeneous polynomial of degree $\delta_X$.
Then, there exists an $r \in \bigwedge^dU$ such that
the directional derivative $h := \frac{\partial f}{\partial r}$ is not the
zero polynomial. By Lemma~\ref{lm:AltIrred} we may assume that
$r=u_1\wedge\ldots\wedge u_d$ for some linearly
independent $u_1,\ldots,u_d \in U$.

We define $Y \subsetneq \Wedge^d$ by
\[ Y(V):=\left\{q \in X(V) ~\middle|~ \forall \phi \in \Hom_\Vec(V,U):
 h((\bigw^d \phi) q)=0 \right\} \]
and note that, by the induction hypothesis, the theorem holds for $Y$. We
define
\[ Z(V) := X(V) \setminus Y(V)\]
and prove that all elements of $Z(V)$ have bounded strength independent of $V$.

\subsection{Shifting}
For $V\in\Vec$ we define
\begin{align*}
P'(V)&:=\bigw^d (U \oplus V)=\bigoplus_{i=0}^d\bigw^{d-i}U \otimes \bigw^i V, \\
X'(V)&:=X(U \oplus V) \subseteq P'(V), \\
Z'(V)&:= \left\{q \in X'(V) ~\middle|~ h((\bigw^d(1_U \oplus 0_{V \to U}))q) \neq 0\right\}.
\end{align*}
We think of $P'(V),X'(V)$ as varieties over $\Wedge^d
U,X(U)$, respectively, via the linear map
$\Wedge^d(1_U\oplus 0_{V \to U})$ and we will write $h$ for
$h\circ \Wedge^d(1_U \oplus 0_{V \to U})$.

\begin{lm} \label{lm:XYZAlt}
We have
$$
\sup_{V\in\Vec}\sup_{q\in X(V)} S(q) =
\sup_{V\in\Vec}\sup\left\{\sup_{q\in Y(V)}S(q),\sup_{q \in
Z'(V)}S(q)\right\}.
$$ \hfill $\square$
\end{lm}

\subsection{Chopping}

\begin{lm} \label{lm:AltChopping}
For $q \in P'(V)$ write $q= q_0 + \ldots + q_d$ with $q_i
\in \Wedge^{d-i}U \otimes \Wedge^iV$. Then $$S(q)\leq \dim U+S(q_d).$$\hfill\qed
\end{lm}

\subsection{Embedding}

Define
$$
Q'(V):=P'(V)/ \bigw^dV = \bigoplus_{i=0}^{d-1}\bigw^{d-i}U \otimes \bigw^iV
$$
and write $\pi(V):P'(V) \to Q'(V)$ for the natural
projection. Let
$$
B(V):=\{q \in Q'(V) \mid h(t) \neq 0\}=\{(q_0,\dots,q_{d-1}) \in Q'(V) \mid h(q_0) \neq 0\},
$$
an open set in $Q'(V)$.  Then $\pi(V)$ maps $Z'(V)$ into
$B(V)$ (and this is a closed embedding by \cite[Lemma 7]{Draisma17} and
Lemma~\ref{lm:AltIrred}).

\subsection{An equivariant map back}\label{sssec:PsiAltTensor}

\begin{lm} \label{lm:PsiAltTensor}
There exists a $\GL(V)$-equivariant polynomial map
$$
\Psi:Q'(V) \to \bigw^dV
$$
such that, for all $q= (q_0,\ldots,q_d) \in Z'(V)$, $q_{d}$
is a scalar multiple of $\Psi(q_0,\ldots,q_{d-1})$.
\end{lm}

\begin{proof} For $x = (x_1,\ldots,x_d) \in (V^*)^d$, let $\phi_{x,j}: V \to U$ be the linear map $v \mapsto x_j(v)u_j$; here $u_1,\ldots,u_d$ are the vectors used in the definition of $h$. Note that $x \mapsto \phi_{x,j}$ is a $\GL(V)$-equivariant linear map $(V^*)^d \to \Hom_{\Vec}(V,U)$. For all $t_1,\dots,t_d\in K$ define the linear map
$$
\Phi_x(t) := \bigw^d(1_U \oplus \textstyle{\sum}_{j=1}^dt_j\phi_{x,j}) : P'(V) \to \Wedge^dU.
$$
Denote the restriction of $\Phi_x(t)$ to the summand $\Wedge^{d-i}U \otimes \Wedge^i V$ by $\Phi_{x,i}$. Note that $\Phi_{x,0}$ is the identity on $\Wedge^dU$ and $\Phi_{x,d}$ is the linear map
\begin{eqnarray*}
\bigw^dV &\to& \bigw^dU\\
v_1\wedge\ldots\wedge v_d &\mapsto& t_1\dots t_d\cdot(\textstyle{\sum}_{j=1}^d\phi_{x,j}(v_1))\wedge\ldots\wedge (\textstyle{\sum}_{j=1}^d\phi_{x,j}(v_d))
\end{eqnarray*}
where the latter is a multiple of $u_1\wedge\ldots \wedge u_d$. Also note that $x\mapsto\Phi_{x,i}$ is a $\GL(V)$-equivariant polynomial map of (ordinary) degree $i$ and that $x\mapsto\Phi_{x,d}$ is multilinear and alternating.

By functoriality of $X$, we have $\Phi_x(t)(X'(V)) \subseteq X(U)$, and for $q = (q_0,\ldots,q_d) \in P'(V)$ we find that
$$
f(\Phi_x(t)(q_0+\ldots+q_d)) = f\left(q_0 + \Phi_{x,1}q_1 + \ldots + \Phi_{x,d-1}q_{d-1} + t_1\dots t_d\cdot(\bigw^d\textstyle{\sum}_{j=1}^d\phi_{x,j})q_d\right),
$$
and this expression vanishes for $q \in X'(V)$. The coefficient of $t_1\dots t_d$ in the Taylor expansion of this expression equals
$$
h(q_0)(x_1\wedge\dots\wedge x_d)q_d + \Psi(x,q_0,\ldots,q_{d-1})
$$
where the function $\Psi: (V^*)^d \times Q'(V) \to K$ is $\GL(V)$-invariant and multilinear in~$(V^*)^d$. We note that for $q \in Z'(V)$, we have $h(q_0) \neq 0$ by definition of $Z'(V)$, and therefore
$$
(x_1\wedge\dots\wedge x_d)q_d = -\frac{1}{h(q_0)}\Psi(x,q_0,\ldots,q_{d-1}).
$$
The map $\Psi$ factors as
\begin{eqnarray*}
(V^*)^d\times Q'(V)&\rightarrow&(V^*)^{\otimes d}\times Q'(V)\\
(x,q')&\mapsto&(x_1\otimes\dots\otimes x_d,q')
\end{eqnarray*}
and a unique $\GL(V)$-equivariant map $(V^*)^{\otimes d}\times Q'(V)\rightarrow K$. If we re-interpret $\Psi$ as a $\GL(V)$-equivariant polynomial map $Q'(V)\to V^{\otimes d}$ and compose the map with the projection $V^{\otimes d}\to \Wedge^d V$, then we get a map $Q'(V)\to\Wedge^dV$ which we also denote by $\Psi$. We see that
$$
d!\cdot q_d = -\frac{1}{h(q_0)}\Psi(q_0,\ldots,q_{d-1}).
$$
for all $q\in Z'(V)$, since
$$
(x_1\wedge\dots\wedge x_d)q_d=(x_1\otimes\dots\otimes x_d)\iota(q_d)
$$
where $\iota$ is the $\GL(V)$-equivariant map
\begin{eqnarray*}
\iota\colon \bigw^dV&\rightarrow& V^{\otimes d}\\
v_1\wedge\dots\wedge v_d&\mapsto&\sum_{\sigma\in S_d}\mathrm{sgn}(\sigma)\cdot v_{\sigma(1)}\otimes\dots\otimes v_{\sigma(d)}
\end{eqnarray*}
\end{proof}

\subsection{Covariants}
A {\em covariant} of $Q'(V)$ (of order $S^d V$, the only order that
we will use) is a $\GL(V)$-equivariant polynomial map $Q'(V) \to
\bigw^d V$. So the map $\Psi$ constructed in~\ref{sssec:PsiAltTensor}
is a covariant. For each integer $i\in[d-1]$, choose a basis $u_{i,1},\dots,u_{i,n_i}$ of $\bigw^{d-i}U$. Then the map
\begin{eqnarray*}
\Phi\colon \bigoplus_{i=0}^{d-i}\left(\bigw^iV\right)^{\oplus n_i}&\to&\bigoplus_{i=1}^{d-1} \bigw^{d-i} U \otimes \bigw^i V\\
(w_{i,j})_{i,j}&\mapsto&\left(\sum_{j=1}^{n_i}u_{i,j}\otimes w_{i,j}\right)_{i=1}^{d-1}
\end{eqnarray*}
is a $\GL(V)$-equivariant isomorphism and the following lemma holds.

\begin{lm} \label{lm:AltCovariants}
Let $\Psi\colon Q'(V)\to \bigw^dV$ be a covariant. Then the composition $\Psi\circ(\id_{\bigwedge^dU},\Phi)$ is of the form
\begin{eqnarray*}
\bigw^dU\oplus\bigoplus_{i=1}^{d-1}\left(\bigw^iV\right)^{\oplus n_i}&\to&\bigw^dV\\
\left(q,(w_{i,j})_{i,j}\right)&\mapsto&\sum_{\substack{\alpha_{i,j}\in\ZZ_{\geq0}\\\sum_{i=1}^{d-1}\sum_{j=1}^{n_i}i\cdot\alpha_{i,j}=d}}p_{\alpha}(q)\cdot \bigwedge_{i=1}^{d-1}\bigwedge_{j=1}^{n_i}\bigwedge_{\ell=1}^{\alpha_{i,j}}w_{i,j}
\end{eqnarray*}
for some polynomial functions $p_{\alpha}\colon \bigw^dU\to K$. \hfill $\square$
\end{lm}

\subsection{Conclusion of the proof}

\begin{proof}[Proof of Theorem~\ref{thm:AltTensor}.]
To bound the strength of elements of $X(V)$ independently of $V$,
by the induction assumption applied to $Y$, it suffices to bound the
strength of elements of $Z(V)$ independently of $V$. Lemma~\ref{lm:XYZAlt}, which focusses the attention to $Z'$, and
Lemma~\ref{lm:AltChopping} together reduce this problem further to bounding the strength of elements $q_d$ for all $(q_0,\dots,q_d)\in Z'(V)$ independently of $V$. Lemma~\ref{lm:PsiAltTensor} shows that such a $q_d$ is contained in the image of a covariant. So Lemma~\ref{lm:AltCovariants} implies that the strength of $q_d$ is bounded by
$$
\sum_{i=1}^{\lfloor d/2\rfloor}\dim\bigw^{d-i}U,
$$
which completes the proof.
\end{proof}

\begin{re} \label{re:ExplicitAlt}
It follows from the induction that $N$ from Theorem~\ref{thm:AltTensor}
can be taken equal to
\begin{equation} \label{eq:BoundAlt}
\dim U + \sum_{i=1}^{\lfloor d/2\rfloor}\dim \bigw^{d-i} U.
\end{equation}
\end{re}

\section{Proof for ordinary tensors} \label{sec:Tensor}

In this section we adapt the proof from Section~\ref{sec:Poly} to a
proof of Theorem~\ref{thm:Tensor} for ordinary tensors.
Throughout this section, we will let $X$
be a closed subset of $T^d$ and $U\in\Vec^d$ a tuple of finite-dimensional
 vector spaces such that $X(U) \subsetneq T^dU$.

\subsection{Homogeneity}
We equip the coordinate ring $K[T^dV]$ with the $d\cdot\ZZz$-grading in which the elements of
$(T^dV)^*$ have degree $d$. For any closed $X\subseteq T^d$
we find, from the fact that $X(V)$ is stable under
$\GL(V)$, that the ideal $\cI(X(V))\subseteq K[T^dV]$ is stable under $\GL(V)$ and in particular
homogeneous. We define
$$
\delta_X:=\min_{f\in\cI(X(U))\setminus\{0\}}\deg f\in\ZZz.
$$

\subsection{Induction}
If $\delta_X=0$, then we find that $X(V)=\emptyset$ for all $V\in\Vec^d$. We
may therefore assume that $\delta_X>0$ and we proceed by
induction, assuming that the theorem holds for all $Y\subseteq T^d$ with
$Y(U)\subseteq T^dU$ and $\delta_Y<\delta_X$.

\subsection{Derivative}
Let $f\in \cI(X(U))\setminus\{0\}$ be a homogeneous polynomial of degree $\delta_X$.
Then, there exists an $r \in T^dU$ such that
the directional derivative $h := \frac{\partial f}{\partial r}$ is not the
zero polynomial and we may assume that $r=u_1\otimes\ldots\otimes u_d$ for some $u_i\in U_i$.

We define the closed subset $Y\subsetneq T^d$ by
$$
Y(V):=\left\{q \in X(V) ~\middle|~ \forall \phi_i \in
\Hom_\Vec(V_i,U_i): h((T^d\phi) q)=0 \right\}
$$
and note that, by the induction hypothesis, the theorem holds for $Y$. We
define
$$
Z(V):=X(V) \setminus Y(V)
$$
and prove that all elements in $Z(V)$ have bounded strength independent of $V$.

\subsection{Shifting}
For $V\in\Vec^d$ we define
\begin{align*}
P'(V)&:=T^d (U \oplus
V)=\bigoplus_{J\subseteq[d]}\left(\bigotimes_{j\in[d]\setminus
J}U_j\otimes\bigotimes_{j\in J}V_j\right), \\
X'(V)&:=X(U \oplus V) \subseteq P'(V), \\
Z'(V)&:=\left\{q \in X'(V) ~\middle|~ h((T^d(1_{U_i} \oplus 0_{V_i \to U_i})_i)q) \neq 0\right\}.
\end{align*}
We think of $P'(V),X'(V)$ as varieties over $T^d
U,X(U)$, respectively, via the linear map $T^d(1_{U_i} \oplus 0_{V_i \to U_i})_i$. We will write $h$ for $h\circ T^d(1_{U_i} \oplus 0_{V_i \to U_i})_i$.

\begin{lm}
We have
$$
Z(U \oplus V) = \bigcup_{g \in \GL(U\oplus V)} g Z'(V).
$$
In particular, $\sup_{q \in Z(U \oplus V)} S(q) = \sup_{q\in Z'(V)} S(q)$.\hfill\qed
\end{lm}

\begin{lm}\label{lm:XYZTensor}
We have
$$
\sup_{V\in\Vec^d}\sup_{q\in X(V)} S(q) =
\sup_{V\in\Vec^d}\sup\left\{\sup_{q\in
Y(V)}S(q),\sup_{q \in Z'(V)}S(q)\right\}.
$$
\hfill\qed
\end{lm}

\subsection{Chopping}
We write $n_{\ell}:=\dim U_{\ell}$ for $\ell=1,\dots,d$.

\begin{lm} \label{lm:TensorChopping}
For $q \in P'(V)$ write $q=\sum_{J\subseteq[d]}q_J$ with
$q_J\in\bigotimes_{j\in[d]\setminus J}U_j\otimes\bigotimes_{j\in J}V_j$. Then
$$
S(q)\leq n_1 + \dots + n_d + S(q_{[d]}).
$$
\hfill\qed
\end{lm}

\subsection{Embedding}

Define
$$
Q'(V):=P'(V)/ T^dV =
\bigoplus_{J\subsetneq[d]}\left(\bigotimes_{j\in[d]\setminus
J}U_j\otimes\bigotimes_{j\in J}V_j\right)
$$
and write $\pi(V):P'(V) \to Q'(V)$ for the natural projection. Let
$$
B(V):=\{q \in Q'(V) \mid h(q) \neq 0\}=\{(q_J)_{J\subsetneq[d]} \in Q'(V) \mid h(q_{\emptyset}) \neq 0\},
$$
an open set in $Q'_U(V)$.  Then $\pi(V)$ maps $Z'(V)$ into $B(V)$.

\subsection{An equivariant map back}\label{sssec:PsiTensor}

\begin{lm} \label{lm:PsiTensor}
There exists a $\GL(V)$-equivariant polynomial
map
$$
\Psi:Q'(V) \to T^dV
$$
such that, for all $q=(q_J)_{J\subseteq[d]} \in Z'(V)$,
$q_{[d]}$ is a scalar multiple of $\Psi(q_J)_{J\subsetneq[d]}$.
\end{lm}
\begin{proof} For $x = (x_1,\ldots,x_d) \in V_1^*\times\dots\times V_d^*$, let $\phi_{x,i}: V_i \to U_i$ be the linear map $v_i \mapsto x_i(v_i)u_i$; here $u_1,\ldots,u_d$ are the vectors used in the definition of $h$. Note that $x \mapsto \phi_{x,i}$ is a $\GL(V)$-equivariant linear map.

For all $t_1,\dots,t_d\in K$ define the linear map $\Phi_x(t) := T^d(1_{U_i} \oplus t_i\phi_{x,i})_i : P'(V) \to T^dU$. The restriction of $\Phi_x(t)$ to the summand
$$
\bigotimes_{j\in[d]\setminus J}U_j \otimes \bigotimes_{j\in J} V_j
$$
equals $\prod_{i\in J}t_i\cdot \Phi_{x,J}$ where $\Phi_{x,J}$ is the map $\bigotimes_{j\in[d]\setminus J}1_{U_j} \otimes \bigotimes_{j\in J} \phi_{x,j}$. Note that $x\mapsto\Phi_{x,J}$ is a $\GL(V)$-equivariant polynomial map of (ordinary) degree $|J|$ and that $x\mapsto\Phi_{x,[d]}$ is multilinear.

By functoriality of $X$, we have $\Phi_x(t)(X'(V)) \subseteq X(U)$, and for $q = (q_J)_{J\subseteq[d]} \in P'(V)$ we find that
$$
f(\Phi_x(t)(q_J)_{J\subseteq[d]}) = f\left(\sum_{J\subsetneq[d]}\left(\prod_{i\in J}t_i\cdot \Phi_{x,J}q_J\right) + t_1\dots t_d\cdot (T^d\phi_x)q_{[d]}\right),
$$
and this expression vanishes for $q \in X'(V)$. The coefficient of $t_1\dots t_d$ in the Taylor expansion of this expression equals
$$
h(q_{\emptyset})(x_1\otimes\dots\otimes x_d)q_{[d]} + \Psi(x,(q_J)_{J\subsetneq[d]})
$$
where the function $\Psi: V_1^*\times\dots\times V_d^* \times Q'(V) \to K$ is $\GL(V)$-invariant and multilinear in $V_1^*\times\dots\times V_d^*$. We note that for $q \in Z'(V)$, we have $h(q_{\emptyset}) \neq 0$ by definition of $Z'(V)$, and therefore
$$
(x_1\otimes\dots\otimes x_d)q_{[d]} = -\frac{1}{h(q_{\emptyset})}\Psi(x,(q_J)_{J\subsetneq[d]}).
$$
The map $\Psi$ factors as the composition of
\begin{eqnarray*}
V_1^*\times\dots\times V_d^*\times Q'(V)&\rightarrow&(V_1^*\otimes\dots\otimes V_d^*)\times Q'(V)\\
(x,q')&\mapsto&(x_1\otimes\dots\otimes x_d,q')
\end{eqnarray*}
and a unique $\GL(V)$-equivariant map $(V_1^*\otimes\dots\otimes V_d^*)\times Q'(V)\rightarrow K$.
We denote the latter map also by $\Psi$, which is now linear in its first argument.
If we re-interpret $\Psi$ as a $\GL(V)$-equivariant polynomial map $Q'(V)\to T^dV$, then
$$
q_{[d]} = -\frac{1}{h(q_{\emptyset})}\Psi(q_J)_{J\subsetneq[d]}.
$$
for all $q\in Z'(V)$.
\end{proof}

\subsection{Covariants}

A {\em covariant} of $Q'(V)$ is a $\GL(V)$-equivariant polynomial
map $Q'(V) \to T^dV$. So the map $\Psi$ constructed
in~\ref{sssec:PsiTensor} is a covariant.
For each non-empty subset $j\subsetneq[d]$, choose a basis $u_{J,1},\dots,u_{J,n_J}$ of $\bigotimes_{j\in[d]\setminus J}U_j$. Then the map
\begin{eqnarray*}
\Phi\colon\bigoplus_{\substack{J\subsetneq[d]\\J\neq\emptyset}}\left(\bigotimes_{j\in J}V_j\right)^{\oplus n_J}&\to&\bigoplus_{\substack{J\subsetneq[d]\\J\neq\emptyset}}\left(\bigotimes_{j\in[d]\setminus
J}U_j\otimes\bigotimes_{j\in J}V_j\right)\\
(w_{J,\ell})_{J,\ell}&\mapsto&\left(\sum_{\ell=1}^{\prod_{j\in[d]\setminus J}\dim U_j}u_{J,\ell}\otimes w_{J,\ell}\right)_J.
\end{eqnarray*}
is a $\GL(V)$-equivariant isomorphism and the following lemma holds.

\begin{lm} \label{lm:TensorCovariants}
Let $\Psi\colon Q'(V)\to T^dV$ be a covariant. Then the composition $\Psi\circ(\id_{T^dU},\Phi)$ is of the form
\begin{eqnarray*}
T^dU\oplus\bigoplus_{\substack{J\subsetneq[d]\\J\neq\emptyset}}\left(\bigotimes_{j\in J}V_j\right)^{\oplus n_J}&\to&T^dV\\
\left(q,(w_{J,\ell})_{J,\ell}\right)&\mapsto&\sum_{\{J_1,\dots,J_k\}\in\mathcal{J}}\sum_{\ell_1\in[n_{J_1}]}\dots\sum_{\ell_k\in[n_{J_k}]}p_{\{J_1,\dots,J_k\},\ell_1,\dots,\ell_k}(q)\cdot \bigotimes_{i=1}^kw_{J_i,\ell_i}
\end{eqnarray*}
for some polynomial functions 
$$
p_{\{J_1,\dots,J_k\},\ell_1,\dots,\ell_k}\colon T^dU\to K
$$
where $\mathcal{J}$ consists of all unordered partitions of $[d]$ into
nonempty sets, i.e., all collections $\{J_1,\dots,J_k\}$ of non-empty
subsets $J_i\subsetneq[d]$ with $J_i\cap J_{i'}=\emptyset$ if $i\neq i'$
and $\bigcup_{i=1}^kJ_i=[d]$. \hfill $\square$
\end{lm}

\subsection{Conclusion of the proof}

\begin{proof}[Proof of Theorem~\ref{thm:Tensor}.]
To bound the strength of elements of $X(V)$ independently of $V$,
by the induction assumption applied to $Y$, it suffices to bound the
strength of elements of $Z(V)$ independently of $V$.
Lemma~\ref{lm:XYZTensor} and Lemma~\ref{lm:TensorChopping} together reduce this problem further to bounding the strength of elements $q_{[d]}$ for all $(q_J)_{J\subseteq[d]}\in Z'(V)$ independently of $V$. Lemma~\ref{lm:PsiTensor} shows that such a $q_{[d]}$ is contained in the image of a covariant. 
So using Lemma \ref{lm:TensorCovariants}, we see that $q_{[d]}$ is a linear 
combination of tensor products of elements $w_{J,\ell}\in \bigotimes_{j\in J}V_j$ where $J$ ranges over non-empty subsets of $[d]$ and $j$ ranges from $1$ to $\prod_{j\in[d]\setminus J}\dim U_j$. Fix an integer $m\in[d]$. Then we note for each $\{J_1,\dots,J_k\}\in\mathcal{J}$ that $m\in J_i$ for some $i\in[k]$. So the strength of $q_{[d]}$ is at most
\begin{eqnarray*}
&&\#\left\{w_{J,\ell}\middle|m\in J\subsetneq[d],\ell\in\left[\prod_{j\in[d]\setminus J}\dim U_j\right]\right\}\\
&\leq &\dim U_m\cdot \sum_{J\subsetneq([d]\setminus\{m\})}\prod_{j\in J}\dim U_j\\
&=&\dim U_m \cdot \left(\prod_{j \in [d]\setminus\{m\}} (\dim U_j+1) - \prod_{j \in [d]\setminus\{m\}} \dim U_j\right)\\
&=&\prod_{j \in [d]} (\dim U_j + 1)- \prod_{j \in[d]\setminus\{m\}} (\dim U_j + 1) - \prod_{j\in [d]} \dim U_j
\end{eqnarray*}
and the latter expression is minimized over $m$ when $\prod_{j \in[d]\setminus\{m\}} (\dim U_j + 1)$ is maximal, which happens exactly when $\dim U_m$ is minimal. This bounds the strength of $q_{[d]}$ independently of $V$.
\end{proof}

\begin{re} \label{re:ExplicitTensor}
It follows from the induction that $N$ from Theorem~\ref{thm:Tensor}
can be taken equal to
\begin{equation} \label{eq:BoundTensor}
n_1+\dots+n_d+ n_m \prod_{j \in [d]\setminus\{m\}}(n_j+1) - \prod_{j \in [d]}n_j,
\end{equation}
where $n_\ell=\dim U_\ell$ and where $n_m \leq n_{\ell}$ for all $\ell \in [d]$.
\end{re}

\begin{re}
The definitions of strength we have used have the following common generalisation:
For integers $0\leq m\leq n$, $d_1,\dots,d_n\in\NN$ with $\sum_id_i\geq2$, and
vector spaces $V_1,\dots,V_n\in\Vec$, the strength $S(q)$ of a composite tensor
\[q\in S^{d_1}V_1\otimes\dots\otimes S^{d_m}V_m\otimes\bigw^{d_{m+1}}V_{m+1}
\otimes\dots\otimes\bigw^{d_n}V_n\]
is the minimal number $k$ of terms in any composition of the form
\[q=r_1 s_1 + \cdots + r_k s_k\]
where
\begin{eqnarray*}
r_i&\in&S^{e_1}V_1\otimes\dots\otimes S^{e_m}V_m\otimes\bigw^{e_{m+1}}V_{m+1}
\otimes\dots\otimes\bigw^{e_n}V_n\\
s_i&\in&S^{d_1-e_1}V_1\otimes\dots\otimes S^{d_m-e_m}V_m\otimes\bigw^{d_{m+1}-e_{m+1}}V_{m+1}
\otimes\dots\otimes\bigw^{d_n-e_n}V_n
\end{eqnarray*}
for suitable $0\leq e_i\leq d_i$ with $(e_1,\dots,e_n)\neq(0,\dots,0),(d_1,\dots,d_n)$.

A version of the main theorem for composite tensors generalising the three versions exists.
A proof of this version can be obtained by modifying the proof in this section.
The most important changes are:
\begin{itemize}
\item[(a)] we must assume that $\cha K=0$ or $\cha K>d_i$ for all $i\in\{1,\dots,n\}$;
\item[(b)] we take $h:=\frac{\partial f}{\partial r}$ where $r=r_1\otimes\dots\otimes r_n$
with $r_i=u_i^{d_i}$, $u_i\in U_i$ for $i\leq m$ and $r_i=u_{i,1}\wedge\dots\wedge u_{i,d_i}$,
$u_{i,j}\in U_i$ for $i>m$; and
\item[(c)] for $i\leq m$ we take $x_i\in V_i^*$ and $t_i\in K$, for $i>m$ we take
$x_{i,j}\in V_i^*$ and $t_{i,j}\in K$ and we let
$\Phi_x(t)=T^d(\Phi_{x_1}^{(1)}(t_1),\dots,\Phi_{x_n}^{(n)}(t_n))$ where $\Phi_{x_i}^{(i)}(t_i)$
is the map from the symmetric case for $i\leq m$ and the map from the alternating case for $i>m$.
\end{itemize}
In addition, the bounds must be adjusted to (more complicated) expressions.
\end{re}

\section{Proof over $\ZZ$.} \label{sec:PolyZ}

\begin{proof}[Proof of Theorems~\ref{thm:PolyZ}, \ref{thm:AltTensorZ} and~\ref{thm:TensorZ}.]
For $P \in \{S^d,\Wedge^d,T^d\}$ we write $P_\ZZ$ for the version over
$\ZZ$ and $P_K$ for the version over $K$. As in the proofs over $K$,
let $U$ be a $\ZZ$-module for which $X_\ZZ(U) \neq P_\ZZ U$, and let $f$
be a nonzero, homogeneous polynomial with integral coefficients vanishing
on $X_\ZZ(U)$. For any field $K$ whose characteristic does not divide all
coefficients of $f$, note that $f$ specialises to a nonzero polynomial vanishing
on $X_K(U_K)$. As the specialisation of $f$ is defined over the prime field,
it follows that if $p=\cha K>0$ and all directional derivatives of $f$ are
zero over $K$, then $f$ is a $p$-th power, and one can replace it
by its $p$-th root. We let $h$ be the partial derivative of $f$ in the direction
of a vector in such a way such that $h$ and $\Psi$ are also defined over the prime
field. Now $\Psi$ has the required properties by Remark \ref{re:XoverZ} since
it has those properties over any infinite field $L\supseteq K$. It follows that
the bound from Remark~\ref{re:Explicit}, \ref{re:ExplicitAlt} or~\ref{re:ExplicitTensor}
applies. Note that this deals with all but finitely many
characteristics. For each of the remaining characteristics $p$, if $K$
is a field of characteristic $p$ such that $X_K \subsetneq P_K$, then
there exists a $U \in \Vec_\ZZ$ and an $f$ as above that specialises to
a nonzero polynomial vanishing on $X_K(U_K)$. So we can take for $N$ the
maximum of finitely many numbers of the form of Remark~\ref{re:Explicit},
 \ref{re:ExplicitAlt} or~\ref{re:ExplicitTensor}.
\end{proof}


\end{document}